\documentclass[12pt, eqno]{article}

\usepackage{amsmath,amssymb,latexsym}

\usepackage{amssymb,latexsym}
\usepackage{amsmath,latexsym}
\usepackage[top=1in,bottom=1in,left=1in,right=1in]{geometry}
\usepackage{amscd}
\usepackage{amsthm}

\newtheorem{theorem}{Theorem}[section]

\newtheorem{corollary}[theorem]{Corollary}

\newtheorem{definition}[theorem]{Definition}

\newtheorem{proposition}[theorem]{Proposition}

\begin{document}

\title{Einstein-Weyl structures on almost cosymplectic manifolds}

\author{\\  Xiaomin Chen
\thanks{
The author is supported by the Science Foundation of China
University of Petroleum-Beijing(No.2462015YQ0604) and partially supported
by  the Personnel Training and Academic
Development Fund (No.2462015QZDX02).
 }\\
{\normalsize College of  Science, China University of Petroleum (Beijing),}\\
{\normalsize Beijing, 102249, China}\\
{\normalsize xmchen@cup.edu.cn}}
\maketitle \vspace{-0.1in}


\abstract{In this article, we study Einstein-Weyl structures on almost cosymplectic manifolds.
First we prove that an almost cosymplectic $(\kappa,\mu)$-manifold is Einstein or cosymplectic if it admits a closed Einstein-Weyl structure or two Einstein-Weyl structures. Next for a three dimensional compact almost $\alpha$-cosymplectic manifold admitting closed Einstein-Weyl structures, we prove that it is Ricc-flat. Further, we show that an almost $\alpha$-cosymplectic admitting two Einstein-Weyl structures is either Einstein or $\alpha$-cosymplectic, provided that its Ricci tensor is commuting. Finally, we prove that a compact $K$-cosymplectic manifold with a closed Einstein-Weyl structure or two special Einstein-Weyl structures is cosymplectic.}
 \vspace{-0.1in}
\medskip\vspace{12mm}

  \vspace{2mm}

 \vspace{2mm}

\section{Introduction}

A Weyl structure $W=(D,[g])$ on a smooth manifold $M$ is a torsion free affine connection $D$ preserving a conformal structure $[g]$. Namely there exists a unique 1-form $\theta$ such that $Dg=-2\theta\otimes g$. The concept of Weyl structure goes back to the work of H. Weyl. He introduced the definition to unify gravitational fields and electromagnetic fields(see \cite{W}). Later on N. Hitchin (\cite{H1}) in studying 3-dimensional minitwistor theory observed that the minitwistor theory can be generalized over a 3-manifold endowed with a Weyl structure satisfying a certain Ricci tensor
condition, called an Einstein-Weyl structure. Refer also to \cite{H2}. A Weyl structure $W=(D, [g])$ is Einstein-Weyl if the symmetrized Ricci tensor is proportional to a metric $g$ representing $[g]$:
\begin{equation}\label{2.4}
  Ric^D(Y,X)+Ric^D(X,Y)=\Lambda g(Y,X),\quad\Lambda\in C^\infty(M).
\end{equation}
Further, if the unique 1-form $\theta$ is closed, then $W$ is said to be a closed Einstein-Weyl structure.  The Einstein-Weyl condition plays a key role in physics, the pure Einstein theory being too strong as a system model for various physical questions. On the contrary, Einstein-Weyl structures appear naturally as the background of the static Yang-Mills-Higgs theory.

On the other hand, almost contact geometry also provides a natural underlying structure
to analyse many problems in physics. For example, Sasakian-Einstein and 3-Sasakian geometry have emerged in the context of dualities of
certain supersymmetric conformal field theories \cite{BG2}, and general almost contact structures have also been used to study special
magnetic fields \cite{U}. Meanwhile,  Matzeu proved that several classes of almost contact manifolds also naturally carry Einstein-Weyl
structures \cite{MA}. Therefore, Einstein-Weyl structures have received a lot of attention in the
frame work of almost contact metric manifold (see \cite{GH,MA2,MA3,N}).

Notice that an Einstein-Weyl structure is a generalization of Einstein metric in terms of affine connection. Recall the Goldberg conjecture \cite{G} that a compact Einstein almost K\"ahler manifold is K\"ahler. The conjecture is true if
  the scalar curvature is non-negative (\cite{S}). As is well known, contact metric manifold can be considered as an odd-dimensional analogue to almost
K\"ahlar manifold. Boyer and Galicki \cite{BG} proposed an odd-dimensional Goldberg conjecture that a compact Einstein $K$-contact manifold is Sasakian and proved that it is true. As a generalization, Ghosh \cite{GH2} and Gauduchon-Moroianu \cite{GM} simultaneously showed that
a compact closed Einstein-Weyl $K$-contact manifold is also Sasakian using different method.

We also remark that another class of almost contact manifold, called \emph{almost cosymplectic manifold}, was also paid many attentions (see a survey \cite{MNY}). The concept was first defined by Goldberg and Yano \cite{GY} as an almost contact manifold whose 1-form $\eta$ and fundamental 2-form $\omega$ are closed.  An almost cosymplectic manifold is said to be \emph{cosymplectic} if in addition the almost contact structure is normal (notice that here we adopt  "cosymplectic" to represent "coK\"ahler" in \cite{MNY}). Concerning cosymplectic geometry, we mention the result that locally conformal cosymplectic manifolds admit a naturally defined
conformally invariant Weyl structure (\cite{MA2}). Later on, Matzeu proved that every
$(2n+1)$-dimensional cosymplectic manifold of constant $\phi$-sectional curvature $c>0$ admits two Ricci-flat Weyl structures where the 1-forms associated to
the metric $g\in[g]$ are $\pm\theta=\pm\lambda\eta$, where $\lambda=\frac{2c}{2n-1}$. More recently, she generalized this result by proving that if a compact cosymplectic manifold $(M,\phi,\xi,\eta,g)$ admits a closed Einstein-Weyl structure $D$, then $M$ is necessarily $\eta$-Einstein (\cite{MA3}).

Recently, Bazzoni-Goertsches \cite{BG} defined a \emph{$K$-cosymplectic manifold}, namely an almost cosymplectic manifold whose Reeb vector field is Killing. In \cite{CP}, in fact it is proved that every compact Einstein $K$-cosymplectic manifold is necessarily cosymplectic.
In addition, Endo \cite{E} defined the notion of \emph{almost cosymplectic $(\kappa,\mu)$-manifold}, i.e. the
curvature tensor of an almost cosymplectic manifold satisfies
\begin{equation}\label{4.1}
  R(X,Y)\xi=\kappa(\eta(Y)X-\eta(X)Y)+\mu(\eta(Y)hX-\eta(X)hY)
\end{equation}
for any vector fields $X,Y$, where $\kappa,\mu$ are constant and $h=\frac{1}{2}\mathcal{L}_\xi\phi$.  As the extension of almost cosymplectic manifold,  Kenmotsu \cite{K} defined the {\it almost Kenmotsu manifold}, which is an almost contact manifold satisfying $d\eta=0$ and $d\omega=2\eta\wedge\omega$.
    Based on this Kim and Pak \cite{KP} introduced the concept of \emph{almost $\alpha$-cosymplectic manifold}, i.e. an almost contact manifold satisfying
$d\eta=0$ and $d\omega=2\alpha\eta\wedge\omega$ for some real number $\alpha$.

Motivated by the above background, in the present paper we first study an almost cosymplectic $(\kappa,\mu)$-manifold and an almost $\alpha$-cosymplectic manifold admitting Einstein-Weyl structures. Finally, we consider a compact $K$-cosymplectic manifold admitting a closed Einstein-Weyl structure. In order to prove our results, we need to recall some definitions and related conclusions on almost cosymplectic manifolds as well as Weyl structures, which are presented in Section 2 and Section 3, respectively. Starting from Section 4, we will state our results and give their proofs.

\section{Almost cosymplectic manifolds}
 Let $M^{2n+1}$ be a $(2n+1)$-dimensional smooth manifold.
An \emph{almost contact structure} on $M$ is a triple $(\phi,\xi,\eta)$, where $\phi$ is a
$(1,1)$-tensor field, $\xi$ a unit vector field, called Reeb vector field, $\eta$ a one-form dual to $\xi$ satisfying
$\phi^2=-I+\eta\otimes\xi,\,\eta\circ \phi=0,\,\phi\circ\xi=0.$
A smooth manifold with such a structure is called an \emph{almost contact manifold}.

A Riemannian metric $g$ on $M$ is called compatible with the almost contact structure if
\begin{equation*}
g(\phi X,\phi Y)=g(X,Y)-\eta(X)\eta(Y),\quad g(X,\xi)=\eta(X)
\end{equation*}
for any $X,Y\in\mathfrak{X}(M)$. An almost contact structure together with a compatible metric
is called an \emph{almost contact metric structure} and $(M,\phi,\xi,\eta,g)$ is called an almost contact metric manifold. An almost contact structure $(\phi,\xi,\eta)$ is said
to be \emph{normal} if the corresponding complex structure $J$ on $M\times\mathbb{R}$ is integrable.

Denote by $\omega$ the fundamental 2-form on $M$ defined by $\omega(X,Y):=g(\phi X,Y)$ for all $X,Y\in\mathfrak{X}(M)$.
An {\it almost $\alpha$-cosymplectic manifold} (\cite{KP,OAM}) is an almost contact metric manifold $(M,\phi,\xi,\eta,g)$ such that the fundamental form $\omega$ and 1-form $\eta$ satisfy $d\eta=0$ and $d\omega=2\alpha\eta\wedge\omega,$ where $\alpha$ is a real number. A normal almost $\alpha$-cosymplectic manifold is called \emph{$\alpha$-cosymplectic manifold}.  $M$ is an {\it almost cosymplectic manifold} if $\alpha=0$ and an {\it almost Kenmotsu manifold} if $\alpha=1$.

Let $M$ be an almost $\alpha$-cosymplectic manifold, we recall that there is an operator
$h=\frac{1}{2}\mathcal{L}_\xi\phi$ which is a self-dual operator.  The Levi-Civita connection
is given by (see \cite{OAM})
\begin{equation}\label{2.4*}
  2g((\nabla_X\phi)Y,Z)=2\alpha g(g(\phi X,Y)\xi-\eta(Y)\phi X,Z)+g(N(Y,Z),\phi X)
\end{equation}
for arbitrary vector fields $X,Y$, where $N$ is the Nijenhuis torsion of $M$.  Then by a simple calculation, we have
\begin{equation}\label{2.2*}
\mathrm{trace}(h)=0,\quad h\xi=0,\quad\phi h=-h\phi,\quad g(hX,Y)=g(X,hY),\quad\forall X,Y\in\mathfrak{X}(M).
\end{equation}

 Using \eqref{2.4*}, a straightforward calculation gives
\begin{equation}\label{2.5}
\nabla_X\xi= -\alpha\phi^2X-\phi hX
\end{equation}
and $\nabla_\xi\phi=0$. Denote by $R$ and $\mathrm{Ric}$ the Riemannian curvature tensor and Ricci tensor, respectively. For an almost $\alpha$-cosymplectic manifold $(M^{2n+1},\phi,\xi,\eta,g)$ the following equations were proved(\cite{OAM}):
\begin{align}
&R(X,\xi)\xi-\phi R(\phi X,\xi)\xi=2[\alpha^2\phi^2X-h^2X]\label{2.6},\\
&(\nabla_\xi h)X= -\phi R(X,\xi)\xi-\alpha^2\phi X-2\alpha hX-\phi h^2X,\label{2.7} \\
&Ric(\xi,\xi)=-2n\alpha^2-\mathrm{trace}(h^2),\label{2.8}\\
&\mathrm{trace}(\phi h)=0,\label{2.9}\\
 &R(X,\xi)\xi=\alpha^2\phi^2X+2\alpha\phi hX-h^2X+\phi(\nabla_\xi h)X\label{2.10}
 \end{align}
for any vector fields $X,Y$ on $M$.

\section{Weyl structures}

Now suppose that $(M,c)$ is a conformal manifold with conformal class $c$. A Weyl connection $D$ in $(M,c)$ is a  torsion-free linear connection which preserves the conformal class $c$. For any metric $g$ in $c$ it carries a 1-form $\theta$, called thee Lee form with respect to $g$, such that $Dg=-2\theta\otimes g$.
It is related to the Levi-Civita connection $\nabla$ by the following relation:
\begin{equation}\label{2.1}
  D_XY=\nabla_XY+\theta(X)Y+\theta(Y)X-g(X,Y)B
\end{equation}
for any vector fields $X,Y$, where $B$ is dual to $\theta$ with respect to $g$. A Weyl structure $W=(D,[g])$ is said to be {\it closed}, resp. {\it exact} if its Lee form is closed, resp. exact with respect to any metric in $c$.

By \eqref{2.1}, a straightforward computation implies the curvature tensor and Ricci tensor of the Weyl connection $D$ are as follows:
\begin{align}
  R^D(X,Y)Z=&R(X,Y)Z+\Sigma_g(X,Y)Z-\Sigma_g(Y,X)Z,\label{2.2}\\
  Ric^D(Y,Z)=&Ric(Y,Z)-2n(\nabla_Z\theta)(Y)+(\nabla_Y\theta)(Z)\label{2.3}\\
  &+(2n-1)\theta(Z)\theta(Y)+(\delta\theta-(2n-1)|\theta|^2)g(Y,Z),\nonumber
\end{align}
where
\begin{align*}
  \Sigma_g(X,Y)Z = &(\nabla_X\theta)(Y)Z+(\nabla_X\theta)(Z)Y-g(Y,Z)\nabla_XB\\
&-g(Y,Z)|\theta|^2X-g(X,Z)\theta(Y )B + \theta(Y)\theta(Z)X
\end{align*}
for $X,Y,Z\in\mathfrak{X}(M)$ and $\delta\theta$ denotes the codifferential of $\theta$ with respect to $g$.

Moreover, the following characterization of closed Weyl connection was proved.
\begin{proposition}[\cite{MA3}] \label{P1}
 Let $(M,\xi,\eta,\phi,g)$ be a $(2n+1)$-dimensional almost contact manifold. Then the Weyl structure $W=(D,[g])$ is closed if and only if $\eta(R^D(X,Y )\xi)=0$ for all vector fields $X,Y\in\mathfrak{X}(M)$.
\end{proposition}
It is well-known that for an almost contact manifold $M$ its tangent bundle $TM$ can be decomposed as $TM=\mathbb{R}\xi\oplus\mathcal{D}$, where $\mathcal{D}=\{X\in TM:\eta(X)=0\}.$ Applying Proposition \ref{P1}, we immediately obtain the following result.
\begin{proposition}\label{P3}
 Let $(M,\xi,\eta,\phi,g)$ be a $(2n+1)$-dimensional almost contact manifold. If the Weyl structure $W=(D,[g])$ is closed, then either $B\in\mathbb{R}\xi$ or $B\in\mathcal{D}$.
\end{proposition}
\begin{proof}
By Proposition \ref{P1}, we obtain from \eqref{2.2} that for all vector fields $X,Y$,
\begin{align*}
  &(\nabla_X\theta)(\xi)\eta(Y)-\eta(Y)\eta(\nabla_XB)-\eta(X)\theta(Y)\eta(B)\\
  &- \Big[(\nabla_Y\theta)(\xi)\eta(X)-\eta(X)\eta(\nabla_YB)-\eta(Y)\theta(X)\eta(B)\Big]=0.
\end{align*}
Here we have used the relation $(\nabla_X\theta)Y=(\nabla_Y\theta)X$ which follows from $d\theta=0.$
Since
\begin{equation*}
  (\nabla_X\theta)(\xi)=\nabla_X(g(B,\xi))-\theta(\nabla_X\xi)=g(\nabla_XB,\xi)=\eta(\nabla_XB),
\end{equation*}
the above relation is simplified as
\begin{align*}
  \Big[-\eta(X)\theta(Y)+\eta(Y)\theta(X)\Big]\eta(B)=0.
\end{align*}
Thus by taking $Y=\xi$, we see that $\theta(X)\eta(B)=0$ for all $X\in\mathcal{D}$, that means that either $B\in\mathcal{D}$ or $B\in\mathbb{R}\xi$.
\end{proof}
A Weyl structure $W=(D,[g])$ is called {\it Einstein-Weyl} if the trace-free component of the symmetric part of $Ric^D$ is identically zero, namely there exists a smooth function $\Lambda$ such that the relation \eqref{2.4} holds.
Thus it follows from \eqref{2.3} and \eqref{2.4} that
\begin{equation}\label{3.5*}
Ric(X,Y)-\frac{2n-1}{2}((\nabla_X\theta)Y+(\nabla_Y\theta)X)+(2n-1)\theta(X)\theta(Y)=\sigma g(X,Y),
\end{equation}
where $\sigma=\delta\theta-(2n-1)|\theta|^2+\frac{\Lambda}{2}$.
Furthermore, if $M$ admits two Einstein-Weyl structures with $\theta$ and $-\theta$, then the following two equations hold for arbitrary
vector fields $X, Y$ in $M$ (Higa \cite{H}) :
\begin{align}
&(\nabla_X\theta)Y+(\nabla_Y\theta)X+\frac{2}{2n+1}\delta\theta g(X,Y)=0,\label{3.6**}\\
&Ric(X,Y)-\frac{r}{2n+1}g(X,Y)=\frac{2n-1}{2n+1}|\theta|^2g(X,Y)-(2n-1)\theta(X)\theta(Y)\label{3.7}.
\end{align}
Here $r$ denotes the scalar curvature of $M$.

Since the Weyl curvature tensor $R^D$ and the Weyl Ricci tensor
$Ric^D$ of closed Einstein-Weyl structures defined on compact manifolds vanish identically (see \cite{Ga}), from \eqref{2.2} and \eqref{2.3}
 we obtain
\begin{align}\label{2.7*}
R(X,Y)Z=&\{(\nabla_Y\theta)Z-\theta(Y)\theta(Z)\}X-\{(\nabla_X\theta)Z-\theta(X)\theta(Z)\}Y\\
&+g(Y,Z)\{(\nabla_XB-\theta(X)B\}-g(X,Z)\{(\nabla_YB-\theta(Y)B\}\nonumber\\
&+|\theta|^2\{g(Y,Z)X-g(X,Z)Y\},\nonumber
\end{align}
\begin{align}\label{3.6*}
  Ric(Y,Z)=&(2n-1)(\nabla_Y\theta)(Z)\\
  &-(2n-1)\theta(Z)\theta(Y)-(\delta\theta-(2n-1)|\theta|^2)g(Y,Z).\nonumber
\end{align}
Moreover, using \eqref{3.6*} we obtain
\begin{equation}\label{3.9}
  (2n-1)\nabla_XB=QX+(2n-1)\theta(X)B+\lambda X,
\end{equation}
where $\lambda=\delta\theta-(2n-1)|\theta|^2$ and $Q$ is the Ricci operator defined by $\mathrm{Ric}(X,Y)=g(QX,Y)$ for any vectors $X,Y$.
Thus from \eqref{3.9}, it is easy to yield
\begin{align}\label{2.10*}
(2n-1)R(X,Y)B=&(\nabla_XQ)Y-(\nabla_YQ)X+\theta(Y)QX\\
           &-\theta(X)QY+\lambda[\theta(Y)X-\theta(X)Y]\nonumber\\
&+(X\lambda)Y -(Y\lambda)X.\nonumber
\end{align}
Taking \eqref{3.9} into account,  the formula \eqref{2.7*} becomes
\begin{align}\label{2.12}
(2n-1)R(X,Y)Z=&\{Ric(Y,Z)-2(2n-1)\theta(Y)\theta(Z)\}X\\
&-\{Ric(X,Z)-2(2n-1)\theta(X)\theta(Z)\}Y\nonumber\\
&+g(Y,Z)\{QX-2(2n-1)\theta(X)B\}\nonumber\\
&-g(X,Z)\{QY-2(2n-1)\theta(Y)B\}\nonumber\\
&+((2n-1)|\theta|^2+2\lambda)\{g(Y,Z)X-g(X,Z)Y\}.\nonumber
\end{align}
Furthermore, putting $Y=Z=\xi$ in \eqref{2.12} gives
\begin{align}\label{5.3}
(2n-1)R(X,\xi)\xi=&\{Ric(\xi,\xi)-2(2n-1)\theta(\xi)^2\}X\\
&-\{Ric(X,\xi)-2(2n-1)\theta(X)\theta(\xi)\}\xi\nonumber\\
&+\{QX-2(2n-1)\theta(X)B\}\nonumber\\
&-\eta(X)\{Q\xi-2(2n-1)\theta(\xi)B\}\nonumber\\
&+((2n-1)|\theta|^2+2\lambda)\{X-\eta(X)\xi\}.\nonumber
\end{align}

\section{Einstein-Weyl structures on almost cosymplectic $(\kappa,\mu)$-manifolds }
In this section we suppose that $(M^{2n+1},\phi,\xi,\eta,g)$ is an almost cosymplectic $(\kappa,\mu)$-manifold, namely the
curvature tensor satisfies \eqref{4.1}. By definition, the equations \eqref{2.2*}-\eqref{2.10} with $\alpha=0$ hold. Furthermore, the following relations are provided (see \cite[Eq.(3.22) and Eq.(3.23)]{MNY}):
\begin{align}
  Q=&2n\kappa\eta\otimes\xi+\mu h,\label{3.1**}\\
  h^2=&\kappa\phi^2\label{3.2}.
\end{align}
Using \eqref{2.2*}, it follows from \eqref{3.1**} that the scalar curvature $r=2n\kappa$ and $Q\xi=2n\kappa\xi.$
By \eqref{3.2}, we find easily that $\kappa\leq0$ and $\kappa=0$ if and only if $M$ is a cosymplectic manifold, thus in the following we always suppose
$\kappa<0$.
\begin{theorem}\label{T1}
A $(2n+1)$-dimensional almost
$(\kappa,\mu)$-cosymplectic manifold admitting a closed Einstein-Weyl structure is an Einstein manifold or a cosymplectic manifold.
\end{theorem}
\begin{proof}
By Proposition \ref{P3}, we know that either $B\in\mathcal{D}$ or $B\in\mathbb{R}\xi$.
Next we consider these two cases respectively.

We first assume that $\theta=f\eta$ for some function $f$. Since $d\theta=0$, by \eqref{2.5}, Eq.\eqref{3.5*} is rewritten as
\begin{align}\label{4.4}
Ric(X,Y)&-(2n-1)(X(f)\eta(Y)-fg(\phi hX,Y))\\
      &+(2n-1)f^2\eta(X)\eta(Y)=\sigma g(X,Y).\nonumber
\end{align}
That is,
\begin{align*}
QX-(2n-1)X(f)\xi+(2n-1)f\phi hX+(2n-1)f^2\eta(X)\xi=\sigma X.
\end{align*}
Applying \eqref{3.1**} in the above formula gives
\begin{equation*}
  2n\kappa\eta(X)\xi+\mu hX-(2n-1)X(f)\xi+(2n-1)f\phi hX+(2n-1)f^2\eta(X)\xi=\sigma X.
\end{equation*}
Replacing $X$ by $hX$ and using \eqref{3.2}, we have
\begin{equation*}
  \sigma hX+(2n-1)f\kappa\phi X=\kappa\mu \phi^2X-(2n-1)hX(f)\xi.
\end{equation*}
Moreover, by taking an inner product of the foregoing relation with $\phi X$ and contracting $X$ over the resulting equation, we get
\begin{equation*}
  -\sigma\;\mathrm{trace}(\phi h)+2n(2n-1)f\kappa=0.
\end{equation*}
Thus the relation \eqref{2.9} shows that $f=0$ and $M$ is Einstein from \eqref{4.4}.

In the following we consider the case where $B\in\mathcal{D}$. In view of \eqref{3.1**} and \eqref{2.5}, the equation \eqref{3.5*} with $Y=\xi$ becomes
\begin{equation*}
(2n\kappa-\sigma)\eta(X)-(2n-1)\theta(\phi hX)=0.
\end{equation*}
Putting $X=\xi$ implies $2n\kappa=\sigma$, so we get $hB=0$ by the above formula. Furthermore it yields from \eqref{3.2} that $\kappa\phi^2B=0$, i.e. $B=0$.

On the other hand, contracting $X$ over \eqref{3.5*} we have
\begin{equation*}
r-(2n-1)\delta\theta+(2n-1)|\theta|^2=(2n+1)\sigma.
\end{equation*}
Because the scalar curvature $r=2n\kappa$, we derive
\begin{equation*}
\delta\theta-|\theta|^2=-\frac{4n^2\kappa}{2n-1}.
\end{equation*}
It comes to a contradiction with $\kappa<0$, hence it is impossible.

Summing up the above discussion, we complete the proof.
\end{proof}
If $M$ admits two Einstein-Weyl structures with $\theta$ and $-\theta$, we immediately prove the following result.
\begin{theorem}\label{T4.2}
A $(2n+1)$-dimensional almost $(\kappa,\mu)$-cosymplectic manifold admitting two Einstein-Weyl structures with $\theta$ and $-\theta$ is either cosymplectic or Einstein.
\end{theorem}
\begin{proof}
By \eqref{3.1**}, the formula \eqref{3.7} with $Y=\xi$ becomes
\begin{align}\label{4.11}
\frac{4n^2\kappa-(2n-1)|\theta|^2}{2n+1}\eta(X)=-(2n-1)\theta(X)\theta(\xi),
\end{align}
which shows that either $B\in\mathbb{R}\xi$ or $B\in\mathcal{D}$ by taking an arbitrary $X\in\mathcal{D}$.

If $B\in\mathbb{R}\xi$, we may set $B=f\xi$ for some smooth function $f$ on $M$. The equation \eqref{3.7} becomes
\begin{equation}\label{4.5}
  Ric(X,Y)=\frac{r+(2n-1)f^2}{2n+1}g(X,Y)-(2n-1)f^2\eta(X)\eta(Y).
\end{equation}
Furthermore, in terms of \eqref{3.6**}, we get
\begin{equation*}
X(f)\eta(Y)-2fg(\phi hX,Y)+Y(f)\eta(X)+\frac{2}{2n+1}\xi(f) g(X,Y)=0.
\end{equation*}
Replacing $X$ by $hX$ and $Y$ by $\phi X$ and contracting $X$ over the resulting equation, we can prove that $f=0$ or $h=0$. Therefore $M$ is an Einstein manifold or a cosymplectic manifold by \eqref{4.5}.

For the case where $B\in\mathcal{D}$, we derive from \eqref{4.11} that $4n^2\kappa=(2n-1)|\theta|^2$. Since $\kappa<0$, it leads to a contradiction.
\end{proof}

\section{Einstein-Weyl structures on almost $\alpha$-cosymplectic manifolds }
In this section we study  an almost $\alpha$-cosymplectic manifold admitting Einstein-Weyl structures. First we consider the case of three dimension.
\begin{theorem}
Let $(M^3,\Phi,\xi,\eta,g)$ be a compact almost $\alpha$-cosymplectic manifold. Suppose that $M$ admits a closed Einstein-Weyl structure. Then $M$ is Ricci-flat.
\end{theorem}
\begin{proof}
As before, by Proposition \ref{P3}, $B\in\mathcal{D}$ or $\theta=f\eta$ where $f=\theta(\xi).$
Next we divide into two cases to discuss.

{\bf Case I.} First we set $\theta=f\eta$ for some function $f$. By \eqref{2.5}, we have
\begin{equation}\label{5.1}
  \nabla_XB=X(f)\xi-f(\alpha\phi^2X+\phi hX).
\end{equation}
Since $d\theta=0$, i.e. $g(\nabla_XB,Y)=g(X,\nabla_YB)$ for all $X,Y\in\mathfrak{X}(M)$, we get $X(f)\eta(Y)=Y(f)\eta(X).$ That means that the gradient vector field $Df=\xi(f)\xi$.
Applying Poincare lemma $d^2=0$, we obtain $g(\nabla_XDf,Y)=g(X,\nabla_YDf)$ for all $X,Y$, thus $\xi(\xi(f))\eta(X)=X(\xi(f))$ by \eqref{2.2*}.
Using \eqref{5.1}, the formula \eqref{3.9} becomes
\begin{equation}\label{5.7}
  QX=X(f)\xi-f(\alpha\phi^2X+\phi hX)-f^2\eta(X)\xi-\lambda X.
\end{equation}
Furthermore, the scalar curvature $r=\xi(f)+2\alpha f-f^2-3\lambda$. In terms of \eqref{5.1} and using \eqref{2.9}, we compute $\lambda=\delta\theta-|\theta|^2=\xi(f)+2\alpha f-f^2$, so we find $r=-2\lambda$.

 On the other hand, it is well known that the curvature tensor of a 3-dimensional Riemannian manifold is given by
\begin{align}\label{5.4*}
  R(X,Y)Z=&g(Y,Z)QX-g(X,Z)QY+g(QY,Z)X-g(QX,Z)Y\\
           &-\frac{r}{2}\{g(Y,Z)X-g(X,Z)Y\}.\nonumber
  \end{align}
Hence substituting \eqref{5.7} into \eqref{5.4*} yields
\begin{align}\label{5.3*}
  R(X,Y)Z=&g(Y,Z)\Big[X(f)\xi-f(\alpha\phi^2X+\phi hX)-f^2\eta(X)\xi\Big]\\
&-g(X,Z)\Big[Y(f)\xi-f(\alpha\phi^2Y+\phi hY)-f^2\eta(Y)\xi\Big]\nonumber\\
&+g\Big(Y(f)\xi-f(\alpha\phi^2Y+\phi hY)-f^2\eta(Y)\xi-\lambda Y,Z\Big)X\nonumber\\
&-g\Big(X(f)\xi-f(\alpha\phi^2X+\phi hX)-f^2\eta(X)\xi-\lambda X,Z\Big)Y.\nonumber
  \end{align}
Putting $Y=Z=\xi$ gives
\begin{align}\label{}
  R(X,\xi)\xi=&-f(\alpha\phi^2X+\phi hX)+(\xi(f)-f^2-\lambda)X\nonumber\\
&-\Big(X(f)-f^2\eta(X)-\lambda \eta(X)\Big)\xi.\nonumber
  \end{align}
Connecting this with \eqref{2.10} implies
\begin{align}\label{5.4**}
&\alpha^2\phi^2X+2\alpha\phi hX-h^2X+\phi(\nabla_\xi h)X\\
=&-f(\alpha\phi^2X+\phi hX)-2\alpha fX-\Big(X(f)-(\xi(f)+2\alpha f)\eta(X)\Big)\xi\nonumber
\end{align}

On the other and, differentiating \eqref{5.7} along $Y$ and using \eqref{2.5}, we conclude
\begin{align*}
 (\nabla_YQ)X=&Y(\xi(f))\eta(X)\xi+\xi(f)g(\nabla_Y\xi,X)\xi+X(f)\nabla_Y\xi\\
&-f\Big(\alpha(\nabla_Y\phi^2)X+(\nabla_Y\phi)hX+\phi(\nabla_Yh)X\Big)\nonumber\\
&-2fY(f)\eta(X)\xi-f^2g(\nabla_Y\xi,X)\xi\nonumber\\
&-f^2\eta(X)\nabla_Y\xi-Y(\lambda)X.\nonumber
\end{align*}
Thus the equation \eqref{2.10*} becomes
\begin{align}\label{5.9}
R(X,Y)B=&Y(\xi(f))\eta(X)\xi+\xi(f)g(\nabla_Y\xi,X)\xi+X(f)\nabla_Y\xi\\
&-f\Big(\alpha(\nabla_Y\phi^2)X+(\nabla_Y\phi)hX+\phi(\nabla_Yh)X\Big)\nonumber\\
&-3fY(f)\eta(X)\xi-f^2g(\nabla_Y\xi,X)\xi-2f^2\eta(X)\nabla_Y\xi\nonumber\\
&-\Big[X(\xi(f))\eta(Y)\xi+\xi(f)g(\nabla_X\xi,Y)\xi+Y(f)\nabla_X\xi\nonumber\\
&-f\Big(\alpha(\nabla_X\phi^2)Y+(\nabla_X\phi)hY+\phi(\nabla_Xh)Y\Big)\nonumber\\
&-3fX(f)\eta(Y)\xi-f^2g(\nabla_X\xi,Y)\xi-2f^2\eta(Y)\nabla_X\xi\Big]\nonumber\\
&+2(X\lambda)Y -2(Y\lambda)X.\nonumber
\end{align}

By comparing \eqref{5.9} and \eqref{5.3*} with $Z=B$, we have
\begin{align*}
  &Y(\xi(f))\eta(X)\xi+\xi(f)g(\nabla_Y\xi,X)\xi+X(f)\nabla_Y\xi\\
&-f\Big(\alpha(\nabla_Y\phi^2)X+(\nabla_Y\phi)hX+\phi(\nabla_Yh)X\Big)\nonumber\\
&-2fY(f)\eta(X)\xi-f^2g(\nabla_Y\xi,X)\xi-f^2\eta(X)\nabla_Y\xi\nonumber\\
&-\Big[X(\xi(f))\eta(Y)\xi+\xi(f)g(\nabla_X\xi,Y)\xi+Y(f)\nabla_X\xi\\
&-f\Big(\alpha(\nabla_X\phi^2)Y+(\nabla_X\phi)hY+\phi(\nabla_Xh)Y\Big)\nonumber\\
&-2fX(f)\eta(Y)\xi-f^2g(\nabla_X\xi,Y)\xi-f^2\eta(Y)\nabla_X\xi\Big]+2(X\lambda)Y -2(Y\lambda)X.\\
=&f\Big(Y(f)-f^2\eta(Y)-\lambda \eta(Y)\Big)X-f\Big(X(f)-f^2\eta(X)-\lambda\eta(X)\Big)Y.\nonumber
  \end{align*}
Now let us put $Y=\xi$, then the above formula is simplified as
\begin{align*}
  &-f\Big(\phi(\nabla_\xi h)X\Big)-2f\xi(f)\eta(X)\xi\nonumber\\
&-\Big[\xi(f)\nabla_X\xi+f\Big(\alpha^2\phi^2X+2\alpha\phi hX-h^2X\Big)\nonumber\\
&-fX(f)\xi-f^2\nabla_X\xi\Big]+2(X\lambda)\xi -2\xi(\lambda)X.\\
=&f\Big(\xi(f)-f^2-\lambda\Big)X+f\eta(X)\Big(f^2+\lambda\Big)\xi.\nonumber
  \end{align*}
Furthermore, by \eqref{5.4**}, the above formula becomes
\begin{align*}
  &-2f\xi(f)\eta(X)\xi-\xi(f)\nabla_X\xi+2fX(f)\xi+2X(\lambda)\xi-2\xi(\lambda)X\\
=&-4f^2\alpha X+2f\eta(X)\Big(f^2+\lambda\Big)\xi.
\end{align*}
Take $X=e_1\in\mathcal{D}$ such that $he_1=\nu e_1$ with $\nu$ being a non-zero function, so the above relation becomes
\begin{align*}
-\xi(f)\nabla_{e_1}\xi+2e_1(\lambda)\xi-2\xi(\lambda)e_1=-4f^2\alpha e_1.
\end{align*}
From \eqref{2.5}, we obtain
\begin{equation*}
  \left\{
     \begin{array}{ll}
       \xi(f)\nu=0,  \\
       \xi(f)\alpha+2\xi(\lambda)=4f^2\alpha.
     \end{array}
   \right.
\end{equation*}
Since $\nu\neq0$ and $\lambda=\xi(f)+2\alpha f-f^2$, the foregoing relations show that $f=0$, so $M$ is Ricci-flat by \eqref{5.7}.

{\bf Case II.} When $B\in\mathcal{D}$, from \eqref{5.3} and \eqref{5.4*} with $n=1$, we follow
\begin{align}\label{5.5*}
  0=&-2\theta(X)B+(|\theta|^2+2\lambda+\frac{r}{2})\{X-\eta(X)\xi\}.
\end{align}
When $X=B$, the above relation becomes
\begin{align*}
  \Big(|\theta|^2-2\lambda-\frac{r}{2}\Big)B=0.
\end{align*}
Thus  $|\theta|^2-2\lambda-\frac{r}{2}=0$ or $B=0$. If $|\theta|^2-2\lambda-\frac{r}{2}=0$, the formula \eqref{5.5*} becomes
\begin{align*}
  0=-\theta(X)B+|\theta|^2\{X-\eta(X)\xi\}.
\end{align*}
Contracting  $X$ over this equation, we also get $|\theta|^2=0$. By \eqref{3.6*}, thus $M$ is Ricci-flat.
\end{proof}

Next we consider the case where $M$ admits two Einstein-Weyl structures $\pm\theta$ and obtain the following result.
\begin{theorem}
Let $(M^{2n+1},\phi,\xi,\eta,g)$ be an almost $\alpha$-cosymplectic manifold. Suppose that $M$ admits two Einstein-Weyl structures with $\pm\theta$.
If the Ricci tensor is commuting, i.e. $\phi Q=Q\phi$, then $M$ is either an Einstein manifold, or an $\alpha$-cosymplectic manifold.
\end{theorem}
\begin{proof}
By \eqref{3.7},  the Ricci operator is expressed as
\begin{equation}\label{5.18}
 QX=\Big(\frac{2n-1}{2n+1}|\theta|^2+\frac{r}{2n+1}\Big)X-(2n-1)\theta(X)B.
\end{equation}
Since the Ricci tensor is commuting,
\begin{equation*}
  \theta(X)\phi B=\theta(\phi X)B.
\end{equation*}
Taking $X=B$ gives $\phi B=0$ or $B=0$. Thus we know that $B=f\xi$, where $f=\theta(\xi).$ In terms of \eqref{3.6**}, we get
\begin{equation*}
X(f)\eta(Y)-2f\alpha g(\phi^2X,Y)-2fg(\phi hX,Y)+Y(f)\eta(X)+\frac{2}{2n+1}\delta\theta g(X,Y)=0.
\end{equation*}
As the proof of Theorem \ref{T4.2}, replacing $X$ by $hX$ and $Y$ by $\phi X$, contracting $X$ over the resulting equation and using \eqref{2.9}, we obtain $f=0$ or $h=0.$ Therefore we complete the proof by \eqref{5.18}.
\end{proof}
\begin{corollary}
Let $(M^{2n+1},\phi,\xi,\eta,g)$ be an almost $\alpha$-cosymplectic manifold. Suppose that $M$ admits two Einstein-Weyl structures with $\pm\theta=\pm f\eta$ for some function $f$, then $M$ either is an Einstein manifold, or an $\alpha$-cosymplectic manifold.
\end{corollary}

\section{ Einstein-Weyl structures on $K$-cosymplectic manifolds}
 Let $M$ be a $(2n+1)$-dimensional almost cosymplectic manifold defined in Section 2, namely the 1-form $\eta$ and the fundamental form $\omega$ are closed  and satisfy $\eta\wedge\omega^n\neq0$ at every point of $M$.

\begin{definition}[\cite{BG}]
An almost cosymplectic manifold $(M,\phi,\xi,\eta,g)$ is called a \emph{$K$-cosymplectic manifold} if the Reeb vector field $\xi$ is Killing.
\end{definition}

For a $K$-cosymplectic manifold $(M,\phi,\xi,\eta,g)$, by Theorem 3.11 in \cite{MNY} we know
\begin{equation*}
  \nabla\xi=\nabla\eta=0.
\end{equation*}
Moreover, it follows from  Theorem 3.29 in \cite{MNY} that
\begin{equation}\label{6.1}
 R(X,Y)\xi=0\quad\text{for all}\;X,Y\in\mathfrak{X}(M).
\end{equation}
That shows that $Q\xi=0$.

In the following we suppose that $M$ admits a closed Einstein-Weyl structure, hence either $B\in\mathbb{R}\xi$ or $B\in\mathcal{D}$ by Proposition \ref{P3}.

If $B\in\mathbb{R}\xi$, we set $\theta=f\eta$ for some smooth function $f$ on $M$.
Then the Ricci tensor \eqref{3.6*} becomes
\begin{align}\label{6.2}
  Ric(Y,X)=(2n-1)X(f)\eta(Y)-(2n-1)f^2\eta(X)\eta(Y)-\lambda g(Y,X).
\end{align}

Using \eqref{6.1}, the formula \eqref{5.3} yields
\begin{align*}
QX=&-\Big((2n-1)|\theta|^2+2\lambda-2(2n-1)f^2\Big)\{X-\eta(X)\xi\}\\
=&-\Big(2\lambda-(2n-1)f^2\Big)\{X-\eta(X)\xi\}.\nonumber
\end{align*}
Combining with \eqref{6.2}, we conclude that
\begin{equation*}
-\Big(\lambda-(2n-1)f^2\Big)X=(2n-1)X(f)\xi-2\lambda\eta(X)\xi.
\end{equation*}
Letting $X\in\mathcal{D}$ we find $X(f)=0$ and $\lambda=(2n-1)f^2$.
From this we see that $Df=\xi(f)\xi$ with $\xi(f)=\frac{2\lambda}{2n-1}.$ On the other hand, since $\lambda=\delta\theta-(2n-1)|\theta|^2$,
$\delta\theta=2\lambda$. Because $\xi$ is Killing, $\delta\theta=\xi(f)$, which yields $\xi(f)=0$, and further $f=0$. That means that $M$ is Ricci-flat, thus
it is cosymplectic in terms of Corollary 3.35 in \cite{MNY}.

Next we assume $B\in\mathcal{D}$, then $\theta(\xi)=0$. From \eqref{5.3} and \eqref{6.1}, the Ricci operator $Q$ is expressed as
\begin{align}\label{6.4}
QX=2(2n-1)\theta(X)B-\Big((2n-1)|\theta|^2+2\lambda\Big)\{X-\eta(X)\xi\}.
\end{align}

Since $\xi$ is Killing and $Q\xi=0$, by \eqref{3.6*} with $Y=Z=\xi$, we see that $\lambda=0,$ i.e. $\delta\theta=(2n-1)|\theta|^2$.
On the other hand, via \eqref{6.4} and \eqref{3.9}, we have
\begin{equation*}
  \nabla_XB=3\theta(X)B-|\theta|^2\{X-\eta(X)\xi\}.
\end{equation*}
Contracting $X$ over the foregoing equation gives $\delta\theta=(3-2n)|\theta|^2$. Hence we get $\theta=0$.

Summing up the above discussion, we actually proved the following conclusion.
\begin{theorem}
Let $(M,\phi,\xi,\eta,g)$ be a compact $(2n+1)$-dimensional K-cosymplectic manifold. Suppose that $M$ admits a closed Einstein-Weyl structure.
Then $M$ is cosymplectic.
\end{theorem}

For a $K$-cosymplectic manifold with two Einstein-Weyl structures with $\pm\theta$, we also have the following conclusion.
\begin{theorem}
Let $(M,\phi,\xi,\eta,g)$ be a $(2n+1)$-dimensional, $2n+1\geq3$,
K-cosymplectic manifold. Suppose that $M$ admits two Einstein-Weyl structures with $\pm\theta$.
Then either $M$ is Ricci-flat, or the scalar curvature is non-positive and invariant along the Reeb vector field $\xi$.
\end{theorem}
\begin{proof}
Since $Q\xi=0$, it follows from \eqref{3.7} that
\begin{align}\label{6.6}
&\Big[-\frac{r}{2n+1}-\frac{2n-1}{2n+1}|\theta|^2\Big]\eta(X)=(2n-1)\theta(X)\theta(\xi).
\end{align}
By taking $X\in\mathcal{D}$, we see that $B\in\mathbb{R}\xi$ or $B\in\mathcal{D}$. When $B\in\mathbb{R}\xi$, as before we set $B=f\xi$, then $\delta\theta=\xi(f)$, so it follows from \eqref{3.6**}
\begin{equation*}
X(f)\eta(Y)+Y(f)\eta(X)+\frac{2\xi(f)}{2n+1} g(X,Y)=0.
\end{equation*}
Replacing $X$ and $Y$ by $\phi X$, we get $\xi(f)=0.$ Further the above relation implies $f=0$. Therefore the equation \eqref{6.6} yields $r=0$ and $M$ is Ricci-flat from \eqref{3.7}.

If $B\in\mathcal{D}$, \eqref{6.6} implies that $r=-(2n-1)|\theta|^2$, and further $QX=-(2n-1)\theta(X)B$ by \eqref{3.7}. So $QB=rB$.
Since $B\in\mathcal{D}$ and $\xi$ is Killing,  taking $X=\xi$ and $Y=B$ in \eqref{3.6**} yields $(\nabla_\xi\theta)B=0$.
Thus we find $\xi(r)=-2(2n-1)(\nabla_\xi\theta)B=0$.
\end{proof}

Since any compact Ricci-flat almost cosymplectic manifold is cosymplectic (see \cite[Corollary 3.35]{MNY}), we conclude immediately from the previous theorem the following corollary.
\begin{corollary}
A compact $K$-cosymplectic manifold admitting two Einstein-Weyl structures with $\pm\theta=\pm f\eta$ for some function $f$ is cosymplectic.
\end{corollary}

\section*{Acknowledgement}
The author would like to thank the referee for the comments and valuable suggestions.

\end{document}